\documentclass[11pt]{amsart}

\newtheorem{theorem}{Theorem}[section]
\newtheorem{proposition}[theorem]{Proposition}

\theoremstyle{rem}
\newtheorem{rem}[theorem]{Remark}
\newtheorem{lemma}[theorem]{Lemma}

\theoremstyle{definition}

\numberwithin{equation}{section}

\begin{document}

\author{E. Liflyand\qquad and \qquad R. Trigub}

\title[Wiener algebras and trigonometric series]
{Wiener algebras and trigonometric series in a coordinated fashion}

\address{Department of Mathematics, Bar-Ilan University, 52900 Ramat-Gan, Israel}
\email{liflyand@math.biu.ac.il}
\address{Department of Mathematics, Donetsk National University, Universitetskaya Str. 24, Donetsk, Ukraine}
\email{roald.trigub@gmail.com}

\begin{abstract}
Let $W_0(\mathbb R)$ be the Wiener Banach algebra of functions representable by the Fourier integrals
of Lebesgue integrable functions. It is proven in the paper that, in particular, a trigonometric series
$\sum\limits_{k=-\infty}^\infty c_k e^{ikt}$ is the Fourier series of an integrable function if and only
if there exists a $\phi\in W_0(\mathbb R)$ such that $\phi(k)=c_k$, $k\in\mathbb Z$. If $f\in W_0(\mathbb R)$,
then the piecewise linear continuous function $\ell_f$ defined by $\ell_f(k)=f(k)$, $k\in\mathbb Z$, belongs to $W_0(\mathbb R)$
as well. Moreover, $\|\ell_f\|_{W_0}\le  \|f\|_{W_0}$. Similar relations are established for more advanced Wiener algebras.
These results are supplemented by numerous applications. In particular, new necessary and sufficient conditions
are proved for a trigonometric series to be a Fourier series and new properties of $W_0$ are established.
\end{abstract}

\maketitle

\section{Introduction}

This paper is devoted to the study of Wiener algebras and of Fourier  series (of functions or measures) in their interrelation.
Finding out many new connections between the two objects, one of which non-periodic and the other periodic, we show that they are
much closer than it has been rated till recently. This gives rise to a variety of new results, for which it is hardly understandable 
how to be on to them without these interrelations,l as well as to new, much shorter proofs of some known ones. To proceed to more detailed
analysis, we begin with basic definitions.

The main instance we are concerned in is the so-called Wiener algebras. We deal with the Wiener algebras of the following forms:
\begin{equation}\label{W0}
W_0=W_0(\mathbb R):=\biggl\{f(x): \ f(x)=\int\limits_{\mathbb R}
e^{ixy}g(y)\,dy,\ g\in L_1(\mathbb R)\biggr\}
\end{equation}
with $\|f\|_{W_0}=\|g\|_{L_1}$;

\begin{equation}\label{mathcalW}
W_1=W_1(\mathbb R):=\biggl\{f(x): \ f(x)=c+f_0, \ f_0\in W_0(\mathbb R),\  c\in \mathbb C\biggr\}
\end{equation}
with $\|f\|_{W_1}=|c|+\|f_0\|_{W_0}$; and

\begin{equation}\label{W}
W=W(\mathbb R):=\biggl\{f(x): \ f(x)=\int\limits_{\mathbb R} e^{ixy}d\mu(y), \mbox{\rm var}\mu<\infty\biggr\}
\end{equation}
with $\|f\|_W=\mbox{\rm var}\mu$. In the most general situation, $\mu$ is a complex-valued Borel measure (charge; for the nice
and detailed presentation, see \cite[Ch.11]{MaPo}).
By $W^+(\mathbb R)$, we denote the subset of $W(\mathbb R)$ with measure $\mu$ positive, that is, consisting of functions positive definite
on $\mathbb R$. The class $W_0^*(\mathbb R)$ is defined by

\begin{equation}\label{W0star} W_0^*(\mathbb R):=\biggl\{f(x): f_0\in W_0(\mathbb R),\ \int_0^\infty
\operatornamewithlimits{ess\, sup}\limits_{|s|\ge t}|g (s)|\,dt<\infty\biggr\}.\end{equation}
It is worth mentioning the classical P\'olya's result on  belonging to $W^+(\mathbb R)$ of any even, bounded, convex and monotone decreasing
to zero function on $[0,\infty)$. This implies the possibility of arbitrary decay to zero at infinity. Moreover, such a function belongs to
$W_0^*(\mathbb R)$ as well (see \cite[Lemma 5]{Tiz}). The reader can find a comprehensive survey of these algebras in \cite{LST}
($W_0^*(\mathbb R)$ is studied in detail in \cite{BLT}); however, certain basics of these spaces will be given here and in the beginning
of the next section for the sake of completeness. For example (see \cite[Th.4.1]{LST}), $W_0(\mathbb R)$, $W_1(\mathbb R)$ and $W(\mathbb R)$
are Banach algebras with point-wise multiplication, and $W_0(\mathbb R)$ is an ideal in $W(\mathbb R)$.
We remark that $W_0(\mathbb R)$ can be extended to $W_1(\mathbb R)$ by adding the unity element to the former.
All these algebras possess the local property. Note that $W_1(\mathbb R)$ is a proper setting for the Wiener property of
simultaneous belonging to the algebra of both $f\ne0$ anywhere and $\frac1f$.

Their periodic prototype, trigonometric series is a classical subject. It will be convenient for us to mostly deal with it in the general complex form

\begin{eqnarray}\label{sercf} \sum\limits_{k=-\infty}^\infty c_k e^{ikt}.\end{eqnarray}
We shall also use the classical form

\begin{eqnarray}\label{fouser} \frac{a_0}2+\sum\limits_{n=1}^\infty(a_n\cos nx+b_n\sin nx),\end{eqnarray}
with the cosine coefficients $a_n$ and the sine coefficients $b_n$. It is well known that the two forms, (\ref{sercf})
and (\ref{fouser}), are equivalent but in certain occasions one of them may be more convenient than the other.
Under certain conditions on the sequence of coefficients $\{c_k\}$, such a series may be (or not be)
the Fourier series of an integrable function, say $f$ (written $f\in L_1(\mathbb T)$, where $\mathbb T=[-\pi,\pi)$).
If yes, we shall say, with slight abuse of terminology, that (\ref{sercf}) is a Fourier series.
For main results, the reader may consult the classical monographs \cite{Zg}, \cite{Br}, \cite{E12}. There are also numerous papers,
to mention some: \cite{Sa}, \cite{Te2}, \cite{AF1}, \cite{BuTa}, \cite{Fr}, and the references therein.
In a wider setting, this series may be the Fourier-Stieltjes series of a Borel measure $\mu$, or, equivalently, of a function of
bounded variation $F$ (written $V_{\mathbb T}(F)<\infty$; we shall write just $V(F)$ if it is clear on which set the total variation
is taken). In the former case the Fourier series is

$$f\sim\sum\limits_{k=-\infty}^\infty \widehat{f}_k e^{ikt}, \qquad \widehat{f}_k =\frac1{2\pi}\int_{\mathbb T}f(t)e^{-ikt}\,dt,$$
while in the latter case the Fourier-Stieltjes series is

$$dF\sim\sum\limits_{k=-\infty}^\infty \widehat{dF}_k e^{ikt}, \qquad \widehat{dF}_k=\frac1{2\pi}\int_{\mathbb T} e^{-ikt}\,dF(t).$$
The latter case reduces to the former one if $F$ is absolutely continuous with respect to the Lebesgue measure, whereas
$dF(t)=f(t)\,dt$. The functions $f$ and $F$ are, in general, complex-valued and $2\pi$-periodic.

The background is as follows. During the more than one century,
there were certain attempts to find necessary and sufficient conditions for the trigonometric series
to be a Fourier series; see \cite{Sa} and certain following attempts.
However, they were not given in terms of the coefficients of the initial trigonometric series nor
they were practical enough. Theorem \ref{Th1}, our first main result, is apparently the most general criterion, much more practical
and is given by means of the coefficients as close as possible. The difference between the latter and the formers
seems to come out from the fact that while the earlier results only approached towards the Wiener algebra language,
Theorem \ref{Th1} exploits this setting explicitly and in its entirety. The following Theorem \ref{Th2}, which is our second
main result, is less general in a sense, but being in terms of a single, quite simple function ($\ell_c$, piecewise linear,
continuous and such that $\ell_c(k)=c_k,k\in \mathbb Z,$) might be more practical in various applications. For the results on the
trigonometric series to be a Fourier series, Theorems \ref{Th1} and \ref{Th2} amount to the following statement:

{\bf Criterion.} {\it For the series (\ref{sercf}) to be a Fourier series, it is necessary and sufficient that $\ell_c\in W_0(\mathbb R)$.}

\noindent Of course, this only gives a flavor of what is obtained and used, Theorem \ref{Th1} opens much more options,
first of all for the algebras $W$ or $W^+$ rather than just $W_0$.

The obtained relations result in two ways for obtaining new and known conditions for the trigonometric series to be a Fourier series or a Fourier-Stieltjes series.
One of them is to search for specific conditions of belonging of the function $\ell_c$ to $W_0$, to $W$, or to another Wiener algebra.
There is a chance for discoveries on this way but a general feeling is that in this case any activity will be equivalent
to that with the sequence $\{c_k\}$. Just because of the structure of $\ell_c$. More promising seems the application of general conditions
for belonging to $W$, $W_0$, etc. to the concrete function $\ell_c$. Till now, for deriving conditions
to be a Fourier series, similar relations between trigonometric series and
corresponding Fourier integrals has been known  for functions of bounded variation. More precisely,

\begin{eqnarray*}\label{onebv} \sup\limits_{|x|\le\pi}\,\biggl|\,\int_{\mathbb
R}f(t)e^{ixt}dt -\sum\limits_{k\in\mathbb Z}f(k)e^{ikx}\biggr|\le 2 V_{\mathbb R}(f).
\end{eqnarray*}
For this result, see, e.g., \cite[4.1.2]{TB} (earlier version with $f$ of compact support is due to Belinsky;
for  general Euler-Maclaurin type formula due to Trigub, see, e.g., \cite[4.1.5]{TB}).

Relating the problem to the Wiener algebras remove all the restrictions and allows one to reconfigure the process
to functional analytic setting.
We immediately exploit these new facilities by deriving a series of sufficient and necessary conditions.
It is of interest that not only some of the results for series are proved by means of Wiener algebra results (Theorem \ref{P1})
but also some new properties of Wiener algebras are derived by means of certain features of the series (Theorem \ref{P2}).
The following description of the structure of the paper gives a map of these applications and hopefully
invites the reader to extend the range of the newly obtained assertions.

In the next section, our main results Theorems \ref{Th1} and Theorem \ref{Th2} are given, with the proofs.
Two mentioned applications,  Theorems \ref{P1} and \ref{P2}, follow immediately. In Section \ref{suf}, a
bunch of sufficient conditions is obtained, both for belonging to Wiener algebras and for the
trigonometric series to be a Fourier series. The former are new first of all for $W$, while the latter have never been obtained
directly. Section \ref{nec} is devoted to various necessary conditions. We conclude with important remarks.

\section{Connections between the Wiener algebras and Fourier series}

As mentioned, our main results establish new connections between the Wiener algebras and trigonometric series.
Roughly speaking, the latter are the Fourier-Stieltjes (Fourier) series if and only if the sequence of the
coefficients can be extended on the whole real axis to a function from the corresponding Wiener algebra.
We note that ${W}$ reduces to $W_0$ in the case of absolutely continuous measures, and to the space of
almost-periodic functions with absolutely convergent Fourier series
if restricted to discrete measures. If a function is in $W(\mathbb R)$, it is uniformly continuous on $\mathbb R$, while a function
$f\in W_0(\mathbb R)$ also vanishes at infinity, $\lim\limits_{|t|\to\infty}f(t)=0$, by the Riemann-Lebesgue lemma.
Outside of a neighborhood of infinity (understood as those $t$ for which $|t|\ge M$) the structure of each of the three algebras is the same. Moreover,
if a function is of bounded variation in a neighborhood of infinity and in $W$, and $f(\infty)=0$, then it is in $W_0$ (see \cite[Th.2]{Tiz}).

\subsection{}
We proceed to our first theorem.
\begin{theorem}\label{Th1} In order that (\ref{sercf}) be the Fourier-Stieltjes series  of $dF$
(the Fourier series), it is necessary and sufficient that a function $\phi\in W$ ($\phi\in W_0$) 
exist such that $\phi(k)=c_k$ for all $k\in \mathbb Z$. By this,

\begin{equation}\label{ndg} V(F)=\min\limits_{\phi} \|\phi\|_W,\end{equation}
where minimum is taken over all such $\phi$. This minimum is attained at

\begin{equation}\label{mina} \phi_0(x)=\int_{\mathbb T} e^{-ixt}\,dF(t).\end{equation}
On the class of entire functions  of exponential type not exceeding $\pi$, all such extensions $\phi$ are of the form

\begin{equation}\label{efetf} \phi(x)=\phi_0(x)+\lambda\sin\pi x,\end{equation}
where $\lambda$ is a number. The function $F$ is monotone increasing if and only if there exists a positive definite $\phi$, that is, $\phi\in W^+$.
\end{theorem}

\begin{proof}  We have

$$
\min\limits_{\phi}\|\phi\|_W\le \|\phi_0\|_W=V(F).
$$
On the other hand, if
$$
\phi(x)=\int_{\mathbb R} e^{-ixt}dF(t),\qquad \|\phi\|_W=V_{\mathbb R}(F),
$$
then, for $k\in \mathbb Z$,

\begin{align*} c_k&=\phi(k)=\int_{\mathbb R} e^{-ikt}\,dF(t)=\sum\limits_{m\in\mathbb Z}\int_{\mathbb T+2m\pi}e^{-ik(t+2m\pi)}\,dF(t)\\
&=\sum\limits_{m\in\mathbb Z}\int_{\mathbb T}e^{-ikt}\,dF(t+2m\pi)=\int_{\mathbb T}e^{-ikt}\sum\limits_{m\in\mathbb Z}dF(t+2m\pi),
\end{align*}
where the Fubini theorem is applied on the final step. Thus,

$$
c_k=\int_{\mathbb T} e^{-ikt}dF_1(t),
$$
with $V(F_1)\le V_{\mathbb R}(F)$, and the given series is the Fourier-Stieltjes series for $F_1$.
Now, for any extension $\phi$,
$$
\|\phi_0\|_W=V(F_1)\le V_{\mathbb R}(F)=\|\phi\|_W.
$$

For the rest, one should take into account Lemma \ref{L2} (given immediately after the proof of the theorem),
  and that $V(F)= \int_{\mathbb R}\,dF(t)$ only for $F$ monotone increasing, and $f\in W^+$ only if

$$ \|f\|_W=f(0)=\int_{\mathbb T} dF(t).$$
The proof is complete.   \end{proof}

\medskip

\begin{lemma}\label{L2} If $f$ is an entire function of exponential type not exceeding $\sigma$, bounded on $\mathbb R$ 
and such that $f(\frac{k\pi}{\sigma})=0$ for all $k\in\mathbb Z$, then $f(z)=\lambda\sin\sigma z$.\end{lemma}

\begin{proof} We may think $\sigma=1$. Then $|f(z)|\le e^{|{\rm Im}\,z|}\sup\limits_x |f(x)|$, and by Bernstein's inequality,
also $|f'(z)|\le e^{|{\rm Im}\,z|}\sup\limits_x |f(x)|$. Therefore, for $|{\rm Im}\,z|\le\delta$, we get

\begin{align*} \sup\limits_{|z-k\pi|\le\frac{\pi}2}\left|\frac{f(z)}{\sin z}\right|&=\sup\limits_{|z-k\pi|\le\frac{\pi}2}\left|\frac{z-k\pi}{\sin (z-k\pi)}\right|\,
\frac1{|z-k\pi|}\,\left|\int_{k\pi}^z f'(\zeta)\,d\zeta\right|\\
&\le \sup\limits_{|z|\le\frac{\pi}2}\left|\frac{z}{\sin z}\right|\,\sup\limits_{\zeta\in [k\pi,z]}|f'(\zeta)|\le
\sup\limits_{|z|\le\frac{\pi}2}\left|\frac{z}{\sin z}\right|\,e^\delta \sup\limits_{x\in\mathbb R} |f(x)|.
\end{align*}
For example, for $-{\rm Im}\,z\ge\delta>0$,

$$|\sin z|=\frac12|e^{iz}-e^{-iz}|\ge\frac12 (e^{-{\rm Im}\,z}-1)=\frac12(e^{|{\rm Im}\,z|}-1).$$
Therefore, the entire function $\frac{f(z)}{\sin z}$ is bounded and thus constant, which completes the proof.
\end{proof}

It is worth mentioning that the constant $\lambda$ in Lemma \ref{L2} depends on the function $f$.

As a consequence of the obtained relation, it follows that if, for instance, all the Fourier-Stieltjes coefficients of $dF$ are zeros, that is,
for each $k\in\mathbb Z$,

$$\int_{-\pi}^\pi e^{-ikt}\,dF(t)=0,$$
then, applying Lemma \ref{L2} to the function $f(z)=\int_{-\pi}^\pi e^{-izt}\,dF(t)$, we derive that $f(z)=\lambda\sin \pi z$.
In other words, only two points of discontinuity on $\mathbb T$ are possible for $F$. Furthermore, if $f\in L_p(\mathbb R)$,
$0<p<+\infty$, or, say, $f\in W^+(\mathbb R)$, then $f\equiv0$.

\medskip
\subsection{}
In order to formulate and prove the next theorem, we shall use both the mentioned in the introduction notation $\ell_c$ and 
mentioned in the abstract $\ell_f$. Both are the piecewise linear continuous functions satisfying $\ell_c(k)=c_k$ and
$\ell_f(k)=f(k)$,  $k\in\mathbb Z$, respectively. More precisely, they are constructed by connecting the values at the integer points
linearly. We shall also call these functions $c$-zigzag and $f$-zigzag, respectively.

\begin{theorem}\label{Th2}  If $f\in W(\mathbb R)$, then $f$-zigzag also belongs to $W(\mathbb R)$, with
$\|\ell_f\|_W\le \|f\|_W$. This inequality is sharp on the whole class. If $f\in W_0(\mathbb R)$, or $f\in W_1(\mathbb R)$,
or $f\in W^+(\mathbb R)$), then $\ell_f$ is such as well. \end{theorem}

\begin{proof} If $f\in W(\mathbb R)$, then, by Theorem \ref{Th1}, we have

$$\sum\limits_{k=-\infty}^\infty f(k)e^{ikx}\sim dF$$
(the Fourier series of a measure or the Fourier-Stieltjes series), and, for example, for any $n\in\mathbb N$,

\begin{align}\label{c1m} \frac1{2\pi}\int_{-\pi}^\pi \left|\sum\limits_{k=-\infty}^\infty f(k)\big(1-\frac{|k|}n\big)_+e^{ikx}\right|\,dx
\le V(F)\le \|f\|_W. \end{align}
We observe that what is integrated is the absolute value of the $(C,1)$-means $\sigma_n(F;x)$.
If, in addition, $f(k)=0$ for $|k|\ge n$, then

\begin{align*} \int_{-\infty}^\infty\ell_f(t)e^{itx}\,dt&=\sum\limits_{k=-n}^n \int_k^{k+1} [f(k) + (t-k) (f(k+1)-f(k))]e^{itx}\,dt\\
&=\sum\limits_{k=-n}^n \biggl\{\frac1{ix}[f(k) + (t-k) (f(k+1)-f(k))]e^{itx}\mid_k^{k+1}\\
&-\frac{f(k+1)-f(k)}{ix}\int_k^{k+1}e^{itx}\,dt\biggr\}\\
&=\sum\limits_{k=-n}^n \frac{f(k+1)-f(k)}{x^2}[e^{i(k+1)x}-e^{ikx}].\end{align*}
Re-indexing the terms corresponding to $c_{k+1}$, we arrive at the estimation of the integrability of the value

$$\frac{(e^{-ix}-1)(e^{ix}-1)}{x^2}\sum\limits_{k=-n}^n f(k)e^{ikx}.$$
In the general case, we thus obtain

$$\int_{-\infty}^\infty\ell_f(t)\big(1-\frac{|t|}n\big)_+e^{itx}\,dt=4\left(\frac{\sin\frac{x}2}{x}\right)^2\sigma_n(F;x).$$
Applying the inverse Fourier transform, we get

\begin{align*} \ell_f(x)\big(1-\frac{|x|}n\big)_+=\frac1{2\pi}\int_{-\infty}^\infty 4\left(\frac{\sin\frac{y}2}{y}\right)^2
\sigma_n(F;y)e^{-ixy}\,dy,\end{align*}                          and

\begin{align*} \|\ell_f(\cdot)\big(1-\frac{|\cdot|}n\big)_+\|_{W_0}&=\frac2{\pi}\int_{-\infty}^\infty|\sigma_n(F;y)|\left(\frac{\sin\frac{y}2}{y}\right)^2\,dy\\
&=\frac2{\pi}\sum\limits_{k=-\infty}^\infty \int_{-\pi+2k\pi}^{\pi+2k\pi}|\sigma_n(F;y)|\left(\frac{\sin\frac{y}2}{y}\right)^2\,dy\\
&=\frac2{\pi}\sum\limits_{k=-\infty}^\infty \int_{-\pi}^{\pi}|\sigma_n(F;y+2k\pi)|\left(\frac{\sin\frac{y+2k\pi}2}{y+2k\pi}\right)^2\,dy\\
&=\frac2{\pi}\int_{-\pi}^{\pi}|\sigma_n(F;y+2k\pi)|\sum\limits_{k=-\infty}^\infty \left(\frac{\sin\frac{y+2k\pi}2}{y+2k\pi}\right)^2\,dy.\end{align*}
In order to calculate the last sum, we apply the Poisson summation formula (see, e.g., \cite[Vol.I, Ch.II, \S 13]{Zg})

$$\sum\limits_{k=-\infty}^\infty g(y+2k\pi)=\frac1{2\pi} \sum\limits_{m=-\infty}^\infty e^{imy}\int_{-\infty}^\infty g(x)e^{-imx}\,dx$$
and the formula (see, e.g., \cite[858.713]{Dw})

$$\int_{-\infty}^\infty \left(\frac{\sin\frac{x}2}{x}\right)^2 e^{-ixy}\,dx=2\int_0^\infty \left(\frac{\sin\frac{x}2}{x}\right)^2 \cos{xy}\,dx
=\pi\big(\frac12-\frac{y}2\big)_+.$$
It follows from these that

$$\sum\limits_{k=-\infty}^\infty \left(\frac{\sin\frac{y+2k\pi}2}{y+2k\pi}\right)^2=\frac1{2\pi}
\int_{-\infty}^\infty \left(\frac{\sin\frac{x}2}{x}\right)^2 \,dx=\frac14,$$
and taking into account (\ref{c1m}), we obtain

$$\|\ell_f(\cdot)\big(1-\frac{|\cdot|}n\big)_+\|_{W_0}=\frac1{2\pi}\int_{-\pi}^{\pi}|\sigma_n(F;y)|\,dy\le V(f).$$
Passing to the limit as $n\to\infty$ (see, e.g., \cite[6.1.4]{TB}, we obtain $\|\ell_f\|_W\le V(f)\le \|f\|_W$.
A similar approach works for $W^+$, just \cite[6.2.2 b)]{TB} is used instead. Moreover, for $f_h(x):=f(hx)$, $h>0$, we have

$$\|\ell_{f_h}\|_{W}\le\|f_h\|_{W}=\|f\|_{W},$$
while in the case of $W^+$ there holds $\|\ell_{f_h}\|=f_h(0)=f(0)=\|f\|$.
By this, $\lim\limits_{h\to 0}\|\ell_{f_h}\|_W=\|f\|_W$.

It remains to prove that $\ell_f\in W_0$ along with $f\in W_0$. If

$$f(x)=\int_{-\infty}^\infty g(y)e^{-iyx}\,dy, \qquad g\in L_1(\mathbb R),$$
and $g: \mathbb R\to \mathbb R$, then

$$f(x)=\int_{-\infty}^\infty |g(y)|e^{-iyx}\,dy-\int_{-\infty}^\infty \big[\,|g(y)|-g(y)\,\big]e^{-iyx}\,dy=f_1(x)-f_2(x),$$
where both $f_1$ and $f_2$ belong to $W^+$. Since it is obvious that $\ell_{f_1}-\ell_{f_2}=\ell_{f}$, we have

$$\|\ell_f\|_{W_0}\le \|\ell_{f_1}\|_{W_0}+\|\ell_{f_2}\|_{W_0}\le 3\|g\|_{L_1}=3\|f\|_{W_0}.$$
For complex valued $g$, we have $g=g_1+ig_2$, where $g_1$ and $g_2$ are real valued. Then $f=\widehat{g}=
\widehat{g_1}+i\widehat{g_2}$, and $\|\ell_f\|_{W_0}\le 6\|f\|_{W_0}$.
This inequality cannot be improved, which follows from the case $f=\ell_f$. The proof is complete.
\end{proof}

For positive definite functions ($W^+$), this theorem has long been known (see \cite[Ch.XIX, Problem 16]{Fel});
for its application in this case, see \cite{Belov}.

\medskip
\subsection{}
We immediately proceed to certain applications. The following theorem shows how essential results
for series can be derived from appropriate results for Wiener algebras.

We will say that a sequence $\{c_k\}$ is of bounded variation, written $\{c_k\}\in bv$, if

$$\sum\limits_{k=-\infty}^\infty |c_k-c_{k+1}|<\infty.$$

\begin{theorem}\label{P1} If (\ref{sercf}) is the Fourier-Stieltjes series, the sequence $\{c_k\}\in bv$,
and $\lim\limits_{k\to\infty}[c_k-c_{-k}]=0$, then

$$\sum\limits_{k=-\infty}^\infty (c_k-\lim\limits_{|k|\to\infty}c_k)e^{ikx}$$
is a Fourier series.  \end{theorem}

\begin{proof} Let us consider the $c$-zigzag. By Theorems \ref{Th1} and \ref{Th2}, we have $\ell_c\in W(\mathbb R)$.
By the $bv$ assumption, $V_{\mathbb R}(\ell_c)<\infty$, therefore $\ell_c$ is representable as the difference of
two monotone bounded functions. Hence, the even part has a limit as $|x|\to\infty$, while the odd part has such a
limit as well, by assumption. Denoting $\lim\limits_{|x|\to\infty}\ell_c(x)=\ell_c(\infty)$, we conclude that by
Theorem \ref{Th2} (see also \cite[Th.2]{Tiz}), we have $\ell_c-\ell_c(\infty)\in W_0(\mathbb R)$, and by Theorem \ref{Th1}
the assertion follows.
\end{proof}

Now, we will show how a property of a Wiener algebra is derived from that for series.

\begin{theorem}\label{P2} 1) For each $f\in W_0(\mathbb R)$, there exists a function $g$, with
$g(|k|)\uparrow+\infty$  as $|k|\to\infty$, for which $gf\in W_0(\mathbb R)$ as well.

2) If a sequence $\{\varepsilon_k\}$, $k\in\mathbb N$, is monotone decreasing to zero, then there is
an even $f\in W_0^*(\mathbb R)$ such that $f(k)\ge  \varepsilon_k$, $k\in\mathbb N$.\end{theorem}

\begin{proof} 1) By Salem's theorem \cite{Sa0} (see also a note to \S 11 of Chapter IV in \cite[Vol. I]{Zg}),
for any Fourier series (\ref{sercf}), there exists a sequence $\{\lambda_k\}_{k=-\infty}^\infty$ such that
$\lambda_{|k|}\uparrow+\infty$ and

$$\sum\limits_{k=-\infty}^\infty \lambda_k c_k e^{ikt}$$
is a Fourier series as well. For each $f\in W_0(\mathbb R)$, the series (\ref{sercf}), with $c_k=f(k)$,
is a Fourier series, by Theorem \ref{Th1}. Therefore, there exists a sequence $\{\lambda_k\}_{k=-\infty}^\infty$ such that
$\lambda_{|k|}\uparrow+\infty$ and

$$\sum\limits_{k=-\infty}^\infty \lambda_k f(k) e^{ikt}$$
is a Fourier series. Again, by Theorem \ref{Th1}, there exists a function $f_\lambda\in W_0(\mathbb R)$ such
that $f_\lambda(k)=\lambda_k f(k)$, as required.

2) The assertion is proved by applying the corresponding fact for the Fourier series (see \cite[Ch.X, \S2]{Br})
and then Theorem \ref{Th1}. \end{proof}

\section{Sufficient conditions}\label{suf}

The following assertion will give sufficient conditions for a function with compact support to belong to $W_0$.
For this, recall certain notions. Let

$$\omega(f;h)=\sup\limits_{0<\delta\le h}|f(x+\delta)-f(x)|$$
be the modulus of continuity of $f$ in $C[-a,a]$, while

$$\omega(f;h)_2=\sup\limits_{0<\delta\le h}\left(\int_{-a}^a |f(x+\delta)-f(x)|^2\,dx\right)^{\frac12}$$
be the modulus of continuity in $L_2$.

\begin{theorem}\label{P3} If $f\in C[-a,a]$ and  ${\rm supp}\,f\subset [-a+\varepsilon, a-\varepsilon]$,
for some $\varepsilon>0$, then for $f\in W_0$ it suffices, and if $|\widehat{f}(\frac {k\pi}{a})|$ is monotone decreasing as
$|k|\to\infty$ is also necessary, that

$$\int_0^1 \frac{\omega(f;t)_2}{\sqrt{t}}\,dt<\infty.$$
If also $V_{[-a,a]}(f)<\infty$, the sufficient condition is

$$\int_0^1 \frac{\sqrt{\omega(f;t)}}{t}\,dt<\infty.$$
The latter condition is sharp on the considered class.    \end{theorem}

\begin{proof} The first step is to apply Theorem \ref{Th1} to the series $\sum\limits_{k=-\infty}^\infty f(k)e^{ikx}$.
We then use the following Lemmas \ref{L4} and \ref{L5}.

\begin{lemma}\label{L4} If a function $f$ is $2a$-periodic and belongs to $C(\mathbb R)$, then in order
that her Fourier series be absolutely convergent, it suffices, and if the absolute values of its Fourier coefficients
are monotone decreasing it is also necessary
$$\int_0^1 \frac{\omega(f;t)_2}{\sqrt{t}}\,dt<\infty.$$
\end{lemma}

\begin{proof} Bernstein and Szasz gave the following sufficient condition for the absolute convergence of Fourier series:

$$\sum\limits_{n=1}^\infty \frac{E_n(f)_2}{\sqrt{n}}<\infty,$$
where $$E_n(f)_2=\min\limits_{d_k}\left(\int_{-a}^a \left|f(x)-\sum\limits_{k=-n}^n d_k
e^{i\frac {k\pi x}{a}}\right|^2\,dx\right)^{\frac12}.$$
On the other hand, Stechkin proved
the necessity of this condition if the absolute values of Fourier coefficients are monotone decreasing. But the series and
the integral in question always converge simultaneously. For all the facts indicated, see \cite[Ch.IX, \S 2 and \S 8]{Br}.
\end{proof}

\begin{lemma}\label{L5} If $f$ is $2a$-periodic and $V_{[-a,a]}(f)<\infty$, then in order
that its Fourier series be absolutely convergent, it suffices

$$\int_0^1 \frac{\sqrt{\omega(f;t)}}{t}\,dt<\infty.$$
This condition is sharp in the sense that if $\int_0^1 \frac{\sqrt{\omega(t)}}{t}\,dt=\infty$, then there exists
$f$ such that $V_{[-a,a]}(f)<\infty$, $ \omega(f;h)\le\omega(h)$ and its Fourier series is not absolutely convergent.
\end{lemma}

The sufficient part is due to Zygmund (see, e.g., \cite[Vol.I, Ch.VI, 3.6]{Zg} or \cite[Ch.IX, \S 3]{Br}),
while the sharpness is due to Bochkarev \cite{B1}. Concerning the sufficiency, it is worth mentioning that
in this case $\omega(f;h)_2\le \sqrt{V_{\mathbb R}(f)h\omega(f;h)}$ (see, e.g., \cite[p.157]{TB}. See also \cite[6.4]{TB}.

The next step is to apply the following lemma.

\begin{lemma}\label{L3} (\cite[Th.7]{Tr74}) If $f\in C(\mathbb R)$, with ${\rm supp}\,f=[-a,a]$, then $f\in W_0$ if and only if
after extending it $2a$-periodically the two functions $f(x)$ and $f_1(x)=xf(x)$ can be expanded in absolutely
convergent Fourier series. If also $f$ vanishes in a neighborhood of $a$ or $-a$, then only one of these two functions
is enough to test the belonging to $W_0$.  \end{lemma}

Finally, applying again Theorem \ref{Th1}, we complete the proof.
\end{proof}

We note that in the case $V_{[-a,a]}(f)<\infty$ it suffices, for $h\to+0$,

$$\omega(f;h)=O\big(\frac1{\ln^{2+\varepsilon}\frac1h}\big).$$
For sufficient conditions on the behavior of a function near infinity, see, for example, the proof of Theorem \ref{P5}. 

\medskip

The following conditions for the trigonometric series to be a Fourier series are completely new.
If $\lim\limits_{|k|\to\infty}c_k=0$, we shall call the sequence $\{c_k\}$ a null-sequence.

\begin{theorem}\label{P5} Let the sequence of the coefficients of (\ref{sercf}) be a null-sequence.

1) If $\sum\limits_{m=1}^\infty \frac1m \left(\sum\limits_{n=m}^\infty \sup\limits_{|k|\ge n}|c_k|
\sup\limits_{|k|\ge n}|c_k-c_{k+1}|\right)^{\frac12}<\infty,$ then (\ref{sercf}) is a Fourier series.

2) If there is  $p\in (0,2]$ such that $\{c_k\}\in l_p$, then (\ref{sercf}) is the Fourier series of an $L_2$ function.
If $\{c_k\}\in l_p$ only for some $p\in (2,+\infty)$, then assuming $\{c_k-c_{k+1}\}\in l_q$, with $q\in\big(0,1+
\frac{1}{p-1}\big)$, we get that (\ref{sercf}) is a Fourier series.
\end{theorem}

\begin{proof} We will prove the validity of each of the two conditions by using the corresponding known
conditions for belonging to $W_0$ to the function $\ell_c$ and applying the criterion given in the introduction.

1) The assertion readily follows from the following sufficient conditions of belonging to $W_0$ in \cite{LiTr1}:
$f$ is locally absolutely continuous on $\mathbb R$, $f(x)=0$ if $|x|\le1$, $\lim\limits_{|x|\to\infty}f(x)=0$, and

$$\int_1^\infty \left(\int_t^\infty \sup\limits_{|u|\ge y\ge1}|f(u)|\,  \operatornamewithlimits{esssup}\limits_{|u|\ge y\ge1}|f'(u)|\,dy
\right)^{\frac12}\,\frac{dt}t<\infty.$$
Applying these conditions to $\ell_c$ completes the proof of 1). In particular, it suffices $\sup\limits_{|k|\ge n}|c_k|
\sup\limits_{|k|\ge n}|c_k-c_{k+1}|=O\big(\frac {1}{n\log^{3+\varepsilon}n}\big)$. This becomes not to be sufficient for
$\varepsilon=0,$ as the example $\sum\limits_{k=1}^\infty \frac{\sin kx}{\log(k+1)}$ shows.

2) Taking into account that the sequence of the coefficients is bounded, we have in the first case that always $\{c_k\}\in l_2$.
For the second case, it is proven in \cite{rae} that if a function is in $L_p$, $1\le p<\infty$, and its derivative
is in $L_q$, $1<q<\infty$, then the function is in $W_0(\mathbb R)$ provided that $\frac1p+\frac1q>1$.
Applying this test to $\ell_c$ leads to the desired conclusion. One should only take into account that if $q\le1$, then the sequence
is in any $l_r$, with $r>1$. Choosing $r$ such that $\frac1p+\frac1r>1$, we complete the proof.
\end{proof}

We note that transferring as above the well known tests for belonging to $W_0$ due to Titchmarsh, Beurling or Carleman
(see \cite[Section 5]{LST}) to Fourier series gives nothing beyond what we already know.

\medskip

The following results are somewhat more restrictive, since they use the traditional assumption for the sequence to be of
bounded variation. We mention that in the preceding theorem there are no a priori claims, except the natural one that the terms
tend to zero. Further, we remind the basics on the Hilbert transforms.
For a (complex-valued) function $f\in L_1(\mathbb R)$, its Hilbert transform $\mathcal{H}f$ is defined by

\begin{align*}\mathcal{H}f(x)&:=\frac{1}{\pi}\,\mbox{\rm (P.V.)}
\int_{\mathbb R} f(x-u)\frac{du}{u}=\frac{1}{\pi}\,\mbox{\rm (P.V.)}
\int_{\mathbb R} \frac{f(u)}{x-u}\,du\\
&=\frac{1}{\pi}\lim\limits_{\delta\downarrow 0}\int_\delta^\infty
\{f(x-u)-f(x+u)\}\frac{du}{u},\quad  x\in\mathbb R.    \end{align*}
As is well known, for $f\in L_1(\mathbb R)$, this limit exists for
almost all $x\in \mathbb R$ but is not necessarily integrable. However, we always have $\|H_f\|_W=\|f\|_W$
(see \cite{LiTr1}). Analogously, for a sequence $c=\{c_k\}$, its discrete Hilbert transform $hc$ can be defined (there are other
equivalent definitions) as the sequence $\{hc_n\}$ so that

\begin{equation}\label{defdht} hc_n=\sum\limits_{k=-\infty}^\infty \frac{c_k}{n+\frac12-k}.\end{equation}
Some details and further references can be found in \cite[Ch.1, 1.3.4]{L2019}.

\begin{theorem}\label{T3h6} Let the sequence of the coefficients of (\ref{sercf}) be a $bv$ null-sequence.
Each of the following additional conditions:

1)  the discrete Hilbert transform  of the sequence $d=\{d_k\}:=\{c_{k+1}-c_k\}$ is summable, written $hd\in l_1$,

\noindent and

2) the discrete Hilbert transform $hc\in bv$,

\noindent guarantees for (\ref{sercf}) to be a Fourier series. \end{theorem}

\begin{proof} The first assertion is straightforward. This result can hardly be considered
as a new one, maybe the language differs from that of preceding works; for numerous versions, see \cite[Ch.3]{IL}, \cite[Ch.8]{L2019}
and references therein. Really new results 1) and 2) in Theorem \ref{P5} are based on recent new conditions for $W_0(\mathbb R)$.
Assertion 2) of the present theorem is also completely new.

1) It is a consequence of a simple and effective sufficient condition for the integrability of a (locally)
absolutely continuous function, vanishing at infinity, to have integrable Fourier transform:
the Hilbert transform of the derivative should be Lebesgue integrable (see, e.g., \cite[Ch.2, 2.5]{L2019}).
Applying it to $\ell_c$ yields the summability of the discrete Hilbert transform of $\{c_{k+1}-c_k\}$ (the derivative of $\ell_c$)
almost immediately. The only thing to be checked is the relation between the discrete Hilbert transform of a summable sequence
$d=\{d_k\}$ and the Hilbert transform of the piecewise constant function $D(t)$ equal $d_k$ on $[k,k+1)$, $k\in\mathbb Z$.

\begin{lemma}\label{htr} Let $d=\{d_k\}$ be a summable null-sequence and let $D(t)$ equal $d_k$ on $[k,k+1)$, $k\in\mathbb Z$. Then

$$\sum\limits_{n=-\infty}^\infty |hd_n|=\int_{\mathbb R} |\mathcal{H}D(x)|\,dx+O\left(\sum\limits_{k=-\infty}^\infty |d_k|\right).$$
\end{lemma}

\begin{proof} The proof uses elementary calculations. One should only take into account that while estimating

$$\sum\limits_{n=-\infty}^\infty \int_n^{n+1}\left| \sum\limits_{k=n-2}^{n+2} \int_k^{k+1}\left[
\frac1{n+\frac12-k}-\frac1{x-t}\right]\,dt\right|\,dx,$$
the inner integral is considered in the principal value sense.
\end{proof}
\medskip

2) This is a kind of an analog of the following Hardy-Littlewood theorem (see, e.g., \cite[Vol.I, Ch.VII, (8.6)]{Zg}):
{\it If a periodic function and its conjugate are both of bounded variation, then their Fourier series converge absolutely.}
First of all, we prove the commutativity of the discrete Hilbert transform and the operation of taking the difference.
Indeed,

\begin{align*} hc_{n+1}-hc_n&=\sum\limits_{k=-\infty}^\infty \frac{c_k}{n+\frac32-k}-\sum\limits_{k=-\infty}^\infty \frac{c_k}{n+\frac12-k}\\
&=\sum\limits_{k=-\infty}^\infty \frac{c_{k+1}}{n+\frac12-k}-\sum\limits_{k=-\infty}^\infty \frac{c_k}{n+\frac12-k}\\
&=\sum\limits_{k=-\infty}^\infty \frac{c_{k+1}-c_k}{n+\frac12-k}=hd_n.\end{align*}
Since it follows from the assumption 2) that the left-hand side is summable, we
have that the discrete Hilbert transform of the sequence of differences is summable, and the assertion follows from 1).
\end{proof}

\section{Necessary conditions}\label{nec}

In this section we appeal to necessary conditions. We begin with a general one.

\begin{lemma}\label{L1} If $f\in W$ is represented by

$$f(x)=\int_{\mathbb R}e^{-ixt}\,dF(t),$$
then, for every $x\in\mathbb R$, the improper integral

$$\int_{\to+0}^{\to+\infty}\frac{f(x+u)-f(x-u)}u\,du=-\pi i\int_{\mathbb R}e^{-ixt}{\rm sign}\,t\,dF(t)$$
converges, and

$$\sup\limits_{x,\delta,M}\left|\,\int_\delta^M \frac{f(x+u)-f(x-u)}u\,du\right|<\infty.$$
The improper integral may not converge absolutely.  \end{lemma}

\begin{proof} We have

$$\int_\delta^M \frac{f(x+u)-f(x-u)}u\,du=-2i\int_{\mathbb R}e^{-ixt}\,dF(t)\int_\delta^M \frac{\sin tu}u\,du,$$
where the inner integral on the right-hand side is bounded by an absolute constant. Therefore, the Lebesgue dominated convergence
theorem allows the passage to the limit as $\delta\to+0$ and $M\to+\infty$. Let us now show that there is $f_1\in W$ such that

\begin{equation}\label{nac} \int_0^1 \frac{|f_1(u)|}u\,du=\frac12 \int_0^1 \frac{|f_1(u)-f_1(-u)|}u\,du=\infty.\end{equation}
For this, we consider the function

$$f_1(x)=\sum\limits_{k=1}^\infty b_k\sin kx, \qquad |x|\le\pi,$$
such that $b_k\ge0$, $\sum\limits_{k=1}^\infty b_k<\infty$, and $\overline{\lim\limits_{k\to\infty}}\, b_k\ln k=\infty$.
If for this function the left-hand side of (\ref{nac}) were finite, then, by Dini's test, the Fourier series
of $f_1(|x|)$ would converge at zero, which is not the case, as shown in \cite[pp. 29-30]{TB}. On the other hand,

$$\|f_1\|_W\le \sum\limits_{k=1}^\infty b_k \left\|\frac{e^{ik\cdot}-e^{-ik\cdot}}2\right\|_W= \sum\limits_{k=1}^\infty b_k<\infty,$$
which completes the proof.
\end{proof}

We mention that on the intervals of monotonicity of $f$, when the difference $f(x+u)-f(x-u)$ is also monotone in $u$,
the convergence of the indicated integral yields $o(\frac1{\ln\frac1u})$ for this difference, as $u\to+0$. This easily implies
 the existence of an even trigonometric series with the coefficients, monotonously tending to zero, which is not a Fourier
 series (the first such example was constructed by Sidon \cite{Si}) as well as the existence of an even function tending to zero
which does not belong to $W(\mathbb R)$.

We observe that if $f$ is odd, continuous and monotone near infinity, it follows from
Lemma \ref{L1} that the integral $\int_{\to0}^{\to\infty}\frac{f(x)}x\,dx$ converges.
This condition is also sufficient. To prove this, the similar condition for the Fourier series (see \cite[Vol.I, Ch.V, 1.14]{Zg}
and Theorem \ref{Th1} should be used. In particular, $f(x)=o\big(\frac1{\ln |x|}\big)$ as $|x|\to\infty$.
The answer is different for even functions from $W_0$, which can decay arbitrarily slowly. It is proved in
\cite{Sal} that in this case  $f(x-1)-f(x+1)=O\left(\frac1{\ln x}\right)$ as $x\to+\infty$.

Furthermore, if the function $f$ is convex (on the right of zero) or it is a difference of two
convex functions, the necessary condition is $\int_0^1 \frac{\omega(f;t)}t\,dt<\infty$. We mention that for the class of
convex functions and more advanced classes, there are asymptotic formulas more general than that in \cite[Ch.IX, \S 6]{Br},
see \cite[6.4.7 and 6.5.9]{TB} and \cite[Part I]{L2019}.

 Taking a concrete function $f(x)=|x|^\alpha\sin\frac{\pi^2}x$ if $|x|\le\pi$ and zero otherwise, we conclude that for $\alpha\in(\frac12,1]$,
we have $\widehat{f}(y)\asymp\frac1{|y|^{\frac{\alpha}2+\frac34}}$ for $|y|\to\infty$ (see \cite[p.117]{BLT}).

\medskip

In Chapter II, \S 9 of \cite{Br} one can find a nice discussion on Salem's necessary
conditions (see \cite{Sa}) for (\ref{fouser}) to be Fourier series. Salem's results say that necessary conditions for
$a_n$ and $b_n$ to be the cosine Fourier coefficients and the sine Fourier coefficients, respectively, are

\begin{eqnarray}\label{na} \lim\limits_{k\to\infty} k\sum\limits_{n=1}^\infty\frac{a_n} {(k+\frac12)^2-n^2}=0\end{eqnarray}
and

\begin{eqnarray}\label{nb} \lim\limits_{k\to\infty} \sum\limits_{n=1}^\infty\frac{nb_n} {(k+\frac12)^2-n^2}=0.\end{eqnarray}
A today glance at these relations allows one to immediately recognize in the expressions under the limit sign in (\ref{na}) and (\ref{nb})
the discrete even and odd Hilbert transforms, respectively. Of course, it was not the case in time of Salem's publication nor of Bary's.

The following arguments give enriched perception of these conditions. If in the odd case the coefficients
$b_k$ decrease monotone (without loss of generality, for all $k\ge1$), then, as is well known (see, e.g., \cite[Vol.I, Ch. V, 1.14]{Zg};
this can easily be derived from the above results), (\ref{fouser}) is a Fourier series if and only if $\sum\frac {b_k}{k}<\infty$.
This follows from Lemma \ref{L1} as well. A different matter is the even case: the coefficients may decay arbitrarily slow.
\medskip

Taking into account a possibility of relation between the Fourier expansion and Hilbert transform,
in much the same manner a necessary condition for a function to belong to Wiener's algebra has been obtained in \cite{Sal}.
It reads as follows: {\it If $f\in W_0(\mathbb R),$ then}

\begin{equation}\label{nW}\lim\limits_{|x|\to\infty}{\mathcal{H}}{f}(x)=0.\end{equation}
In fact, in \cite{Sal} a correct but somewhat misleading result
was formulated, with the so-called modified Hilbert transform rather than $\mathcal{H}f$.
The idea was that for a bounded (and continuous, which is the case for $f\in W_0(\mathbb R)$) function,
the usual Hilbert transform may not exist. However, for $f\in W_0(\mathbb R)$, its Hilbert transform
exists everywhere (see, e.g., \cite{LiTr1}). We shall now derive Salem's conditions from (\ref{nW})
in the following form.

\begin{proposition}\label{salfw} If (\ref{sercf}) is a Fourier series, then
the discrete Hilbert transform $hc_n$ exists for all $n=0,\pm1, \pm2,...,$ and

\begin{eqnarray}\label{sc} \lim\limits_{|n|\to\infty} hc_n=\lim\limits_{|n|\to\infty}
\sum\limits_{k=-\infty}^\infty \frac{c_k}{n+\frac12-k}=0.\end{eqnarray}
\end{proposition}

\begin{rem}\label{r1} {\rm In words, the necessary condition for the trigonometric series to be a Fourier series is the
same as that for a function to be in the Wiener algebra: the Hilbert transform vanishes at infinity.
Of course, in the case of series the transform is discrete, it is defined by (\ref{defdht}).
If one decides to represent (\ref{sercf}) in the form (\ref{fouser}), then (\ref{sc}) easily reduces to (\ref{na})
and (\ref{nb}). Similarly, in the case of $f$ even or odd, the condition in (\ref{nW}) reduces to

$$\lim\limits_{x\to\infty} x\int_0^\infty \frac{f(t)}{x^2-t^2}\,dt=0 \quad {\rm and}\quad
\lim\limits_{x\to\infty} \int_0^\infty \frac{tf(t)}{x^2-t^2}\,dt=0,$$
respectively (see again \cite[Ch.1, 1.3.4]{L2019}).}
\end{rem}

\begin{proof} With the criterion, given in the introduction, in hand, we continue with $\ell_c$. To understand what follows from the fact that
the Hilbert transform of $\ell_c$ vanishes at infinity, we consider the summands of the form

\begin{equation}\label{ders} \int_k^{k+1} [c_k + (t-k) (c_{k+1}-c_k)]\frac{dt}{x-t}.\end{equation}
Taking $x=n+\frac12$, we first estimate

\begin{equation}\label{ders1} \int_k^{k+1} (t-k)(c_{k+1}-c_k)\frac{dt}{x-t}.\end{equation}
Substituting $k+1\to k$ in the sum related to $c_{k+1}$ and then $t-1\to t$ in the integral
over $[k-1,k]$, we reduce summation of (\ref{ders1}) to

\begin{align*} &\sum\limits_{k=-\infty}^\infty c_k \int_k^{k+1} \frac{x-k+1}{(x-t)(x-t-1)}\,dt=
\sum\limits_{k=-\infty}^\infty \frac{c_k}{x-k}\\ +&\sum\limits_{k=-\infty}^\infty c_k
\int_k^{k+1} \frac{(t-k)(2x-k-t+1)}{(x-t)(x-k)(x-t-1)}\,dt.\end{align*}
The integrals for $k-1<x<k+1$ can be treated separately, taking into account that they are
understood in the principal value sense (if needed). With this in hand, we split the sum as

$$\left(\sum\limits_{k=-\infty}^{[n/2]}+\sum\limits_{k=[n/2]+1}^\infty\right) c_k\int_k^{k+1}
\frac{(t-k)(2x-k-t+1)}{(x-t)(x-k)(x-t-1)}\,dt.$$
The first sum here can be estimated by

$$\max|c_k|\sum\limits_{k=-\infty}^{[n/2]}\frac1{(n-k)^2}\le C\frac{ \max|c_k|}n,$$
which tends to zero as $n\to\infty$. The second one is bounded by

$$C\max\limits_{k>[n/2]}|c_k|,$$
which also tends to zero as $n\to\infty$. Here $C$ denotes absolute constants, not necessarily the same. Since

$$\int_k^{k+1} c_k \frac{dt}{x-t}=\frac{c_k}{x-k}+c_k\int_k^{k+1} \frac{t-k}{(x-t)(x-k)}\,dt,$$
and summing the integrals is treated as that above, we arrive at

$$\mathcal{H}\ell_c(x)=2\sum\limits_{k=-\infty}^\infty \frac{c_k}{x-k}+o(x).$$
This completes the proof.
\end{proof}

\section{Concluding remarks}

In fact, Fourier transform approaches are applied to the study of Fourier series for a long time, especially
in the problems of summability and the rate of convergence of the linear means of Fourier series; see \cite[Ch.7-9]{TB}.
One more area of application of such methods is the study of the almost everywhere convergence with indication
of the set of convergence. For example, the criterion of the summability at the Lebesgue points reduces to
checking whether the multiplier function which generates the summability method belongs to $W_0^*$ (\cite[8.1.3]{TB}; see the definition in (\ref{W0star})).
Recently, a general result for the summability at all the points of the differentiability of indefinite integral
has been proven in \cite{Tr16} in terms of belonging of certain functions to a Wiener algebra. For instance, it follows from this that
the Ces\'aro means converge at each Lebesgue point of an integrable function but may diverge at the differentiability points.

\end{document}